\documentclass[a4paper,11pt]{article}
\usepackage{amsmath}
\usepackage{amssymb}
\usepackage{theorem}
\usepackage{pstricks}
\usepackage{euscript}
\usepackage{epic,eepic}
\usepackage{graphicx}
\PassOptionsToPackage{normalem}{ulem}
\usepackage{amsmath}
\usepackage{amssymb}
\usepackage{theorem}
\usepackage{pstricks}
\usepackage{euscript}
\usepackage{epic,eepic}
\usepackage{graphicx}
\usepackage{setspace}
\PassOptionsToPackage{normalem}{ulem}

\newcommand{\menge}[2]{\big\{{#1} \;|\; {#2}\big\}}

\newcommand{\emp}{\ensuremath{{\varnothing}}}

\newcommand{\scal}[2]{\left\langle{#1}\mid {#2} \right\rangle}

\newcommand{\vuo}{\ensuremath{\mbox{\footnotesize$\square$}}}

\newcommand{\HH}{\ensuremath{\mathcal H}}
\newcommand{\GG}{\ensuremath{\mathcal G}}

\newcommand{\RR}{\ensuremath{\mathbb R}}

\newcommand{\NN}{\ensuremath{\mathbb N}}

\newcommand{\dom}{\ensuremath{\operatorname{dom}}}

\newcommand{\prox}{\ensuremath{\operatorname{prox}}}

\newcommand{\TT}{\ensuremath{{\mathbf T}}}

\newcommand{\ran}{\ensuremath{\operatorname{ran}}}
\newcommand{\zer}{\ensuremath{\operatorname{zer}}}
\newcommand{\gra}{\ensuremath{\operatorname{gra}}}

\newcommand{\sss}{r}
\newcommand{\xx}{x}

\newcommand{\pp}{p}

\newcommand{\yy}{y}

\newcommand{\rr}{r}

\newcommand{\BB}{B}

\newcommand{\E}{\ensuremath{\mathsf{E}}}

\newcommand{\AAA}{A}

\newcommand{\QQ}{B}

\newcommand{\Id}{\ensuremath{\operatorname{Id}}}

\newcommand{\weakly}{\ensuremath{\rightharpoonup}}

\newcommand{\pinf}{\ensuremath{+\infty}}

\newtheorem{theorem}{Theorem}[section]
\newtheorem{lemma}[theorem]{Lemma}

\newtheorem{corollary}[theorem]{Corollary}

\newtheorem{definition}[theorem]{Definition}
\theoremstyle{plain}{\theorembodyfont{\rmfamily}
}
\theoremstyle{plain}{\theorembodyfont{\rmfamily}
}
\theoremstyle{plain}{\theorembodyfont{\rmfamily}
\newtheorem{algorithm}[theorem]{Algorithm}}
\theoremstyle{plain}{\theorembodyfont{\rmfamily}
}
\theoremstyle{plain}{\theorembodyfont{\rmfamily}
\newtheorem{problem}[theorem]{Problem}}
\theoremstyle{plain}{\theorembodyfont{\rmfamily}
\newtheorem{remark}[theorem]{Remark}}
\theoremstyle{plain}{\theorembodyfont{\rmfamily}
}

\definecolor{labelkey}{rgb}{0,0.08,0.45}
\definecolor{refkey}{rgb}{0,0.6,0.0}
\definecolor{Brown}{rgb}{0.45,0.0,0.05}
\definecolor{dgreen}{rgb}{0.00,0.49,0.00}
\definecolor{dblue}{rgb}{0,0.08,0.75}
\numberwithin{equation}{section}
\renewcommand\theenumi{(\roman{enumi})}
\renewcommand\theenumii{(\alph{enumii})}

\title{ Convergence analysis of the stochastic reflected forward-backward splitting algorithm}

\author{ Nguyen Van  Dung$^1$ and 
B$\grave{\text{\u{a}}}$ng C\^ong V\~u$$\\[5mm]
 \\
$^1$ Department of Mathematics, University of Transport and Communications,\\ 3 Cau Giay Street, Hanoi, Vietnam\\
 dungnv@utc.edu.vn;  bangcvvn@gmail.com
 } 
\begin{document}
\maketitle
\begin{abstract}
    We propose and analyze the convergence of a novel stochastic algorithm for solving monotone inclusions that are the sum of a maximal monotone operator and a monotone, Lipschitzian operator. The propose algorithm requires only unbiased estimations of the Lipschitzian operator.
     We obtain the rate $\mathcal{O}(log(n)/n)$ in expectation for the strongly monotone case, as well as almost sure convergence for the general case. 
     Furthermore, in the context of  application to convex-concave saddle point problems, we derive the rate of the primal-dual gap. In particular,
     we also obtain $\mathcal{O}(1/n)$ rate convergence of the primal-dual gap in the deterministic setting.
\end{abstract}

{\bf Keywords:} 
monotone inclusion, stochastic optimization, stochastic error,
monotone operator,
operator splitting,
 reflected method, Lipschitz,
composite operator,
duality,
primal-dual algorithm, ergodic convergence

{\bf Mathematics Subject Classifications (2010)}: 47H05, 49M29, 49M27, 90C25 
\section{Introduction}

A wide class of problems in monotone operator theory, variational inequalities, convex optimization, image processing, machine learning, reduces to the problem of solving monotone inclusions
involving Lipschitzian operators; see \cite{livre1,siop2,Buicomb,Luis18,plc6,sva2,Malitsky18b,Vu19,pesquet,
Tseng00,Bang16} and the references therein. In this paper, we revisit the generic monotone inclusions of 
%\begin{problem}
%\label{p:primal-dual}
finding a zero point of the sum of a maximally monotone operator $A$ and a monotone, $\mu$-Lipschitzian operator $B$, acting on a real separable Hilbert space $\HH$, i.e., 
\begin{equation}\label{p:primal-dual}
    \text{Find $\overline{x}\in\HH$ such that}\; 0 \in (A+B)\overline{x}.
\end{equation}
The first splitting method proposed for solving problem was in \cite{Tseng00} which is now known as the forward-backward-forward splitting method (FBFSM).  Further investigations of this method lead to a new primal-dual splitting method in \cite{siop2} where $B$ is a linear monotone skew operator in suitable product spaces.
A main limitation of the FBFSM is their two calls of $B$ per iteration. This issue was recently resolved in \cite{Malitsky18b} in which the forward reflected backward splitting method (FRBSM) was proposed, namely,
\begin{equation}\label{e:FRB}
\gamma \in \left]0,+\infty \right[, \quad  x_{n+1} = (\Id +\gamma A)^{-1}(x_n-2\gamma Bx_n+ \gamma Bx_{n-1}).
\end{equation}
An alternative approach to overcome this issue was in \cite{sva2} where the reflected forward backward splitting method (RFBSM) was proposed: 
\begin{equation}\label{e:RFB}
\gamma \in \left]0,+\infty \right[, \quad  x_{n+1} = (\Id +\gamma A)^{-1}(x_n-\gamma B(2x_n-x_{n-1})).
\end{equation}
It is important to stress that the methods FBFSM, RFBSM and RFBSM are limited to the deterministic setting. The stochastic version of FBFSM was investigated in \cite{Bang16}
and recently in \cite{Cui20}. Both works \cite{Bang16} and \cite{Cui20} requires two 
stochastic approximation of $B$. While, a stochastic version of FRBSM was also considered in \cite{Malitsky18b} for the case when $B$ is a finite sum. However, it remains require to evaluate the operator $B$.

\noindent The objective of this paper is to avoid these above limitations of \cite{Bang16,Cui20,Malitsky18b} by considering the stochastic counterpart of \eqref{e:RFB}.
At each iteration, we use only one unbiased estimation of $B(2x_{n}-x_{n-1})$ and hence the resulting algorithm shares the same structure as the standard stochastic forward-backward splitting \cite{Combettes2015,Combettes2016,Jota2}. However, it allows to solve a larger class of problems involving non-cocoercive operators.

In Section \ref{s:nota}, we recall  the basic notions  in convex analysis and monotone operator theory as well as the probability theory,
and establish the results which will be used in the proof of the convergence of the proposed method. We present the proposed method and derive the almost sure convergence, convergence in expectation in Section \ref{s:main}. In the last section, we will further apply the proposed algorithm to the convex-concave saddle problem involving the infimal convolutions, and  establish   the  rate of the ergodic convergence of the primal-dual gap.
%\end{problem}
\section{Notation and Background}
\label{s:nota}
 Let $\mathcal{H}$ be a separable real Hilbert space endowed with the inner product $\scal{.}{.}$ and the associated norm $\|.\|$. Let $(x_n)_{n\in\NN}$ be a sequence in $\mathcal{H}$, and $x\in\HH$.
We denote the strong convergence and the weak convergence of $(x_n)_{n\in\NN}$ to $x$ by $x_n\to x$ and  $x_n\weakly x$, respectively.
\begin{definition}
    Let $A\colon\HH\to 2^{\HH}$ be a set-valued operator.
    \begin{enumerate}
  \item  The domain of $A$ is denoted by $\dom(A)$ that 
is a set of all $x\in\HH$ such that $Ax\not= \emp$.
\item The range of $A$ is $\ran(A) = \menge{u\in\HH}{(\exists x\in\HH) u\in Ax }$. 
\item 
The graph of $A$ is 
$\gra(A) = \menge{(x,u)\in\HH\times\HH}{u\in Ax}$. 
\item 
The inverse of $A$ 
is $A^{-1}\colon u \mapsto \menge{x}{ u\in Ax}$. 
\item 
The zero set of $A$ is $\zer(A) = A^{-1}0$.
\end{enumerate}
\end{definition}
\begin{definition}  We have the following definitions:
\begin{enumerate}
\item We say that $A\colon\HH\to 2^\HH$ is  monotone if 
\begin{equation}\label{oioi}
\big(\forall (x,u)\in \gra A\big)\big(\forall (y,v)\in\gra A\big)
\quad\scal{x-y}{u-v}\geq 0.
\end{equation}
\item We say that $A\colon\HH\to 2^\HH$ is maximally monotone if it is monotone and  there exists no monotone operator $B$ such that $\gra(B)$ properly contains $\gra(A)$.
\item We say that  $A\colon\HH\to 2^\HH$  is $\phi_A$-uniformly monotone, at $x\in\dom (A)$, if there 
 exists an increasing function $\phi_A\colon\left[0,\infty\right[\to \left[0,\infty\right]$ that vanishes only at $0$ such that 
 \begin{equation}\label{unif}
\big(\forall u\in Ax\big)\big(\forall (y,v)\in\gra A\big)
\quad\scal{x-y}{u-v}\geq \phi_A(\|y-x\|).
\end{equation}
If $\phi_A = \nu_A|\cdot|^2$  for some $\nu_A\in \left]0,\infty\right[$, then we say that $A$ is $\nu_A$-strongly monotone.
\item The resolvent of $A$ is 
\begin{equation}
 J_A=(\Id + A)^{-1},\;  
\end{equation}
where $\Id$ denotes the identity operator on $\HH$.  
\item A single-valued operator $B\colon\HH\to\HH$ is $1$-cocoercive or firmly nonexpasive if
\begin{equation}
(\forall (x,y)\in\HH^2)\; \scal{x-y}{Bx-By} \geq \|Bx-By\|^2.
\end{equation}  
\end{enumerate}
\end{definition}
Let $\Gamma_0(\mathcal{H})$ be the class of proper lower semicontinuous convex function from $\mathcal{H}$ to $\left ]-\infty,+\infty\right]$. 
\begin{definition}
For $f \in \Gamma_0(\mathcal{H})$:
\begin{enumerate}
\item The proximity operator of $f$ is
\begin{align}
\text{prox}_f:\mathcal{H} \to \mathcal{H}: x \mapsto \underset{y \in \mathcal{H}}{\text{argmin}}\big(f(y)+\frac{1}{2}\|x-y\|^2\big).
\end{align}
\item The conjugate function of $f$ is
\begin{align}
f^*:a \mapsto \underset{x \in \mathcal{H}}{\text{sup}} \big( \scal{a}{x}-f(x) \big).
\end{align}
 \item The infimal convolution of the two functions $\ell$ and $g$ from $\HH$ to $\left]-\infty,+\infty\right]$ is 
 \begin{equation}
 \ell\;\vuo\; g\colon x \mapsto \inf_{y\in\HH}(\ell(y)+g(x-y)).
\end{equation}
 \end{enumerate}
\end{definition}
Note that  $\text{prox}_f=J_{\partial f}$, let $x \in \mathcal{H}$ and set $p=\text{prox}_f x$, we have
\begin{align}
(\forall y \in \mathcal{H})\ \ f(p)-f(y) \le \scal{ y-p}{p-x},
\end{align}
and that 
\begin{align}
(\forall f \in \Gamma_0(\mathcal{H})) \ \ (\partial f)^{-1}=\partial f^*.
\end{align}
%%%%%%%%%%%%%%%%%%%%%%%%%%%%%%%%%%%%%%%%%%%%%%%%%%%%%%%%%
%\begin{lemma}\label{l2}{\rm(\cite{opial67})} Let $C$ be a nonempty set of $H$ and $\{x_n\}$ be a sequence in $H$ such that the following two conditions hold:
%	\begin{enumerate}
%	\item[\upshape(i)] for every $x\in C$, $\lim_{n\to\infty}\|x_n-x\|$ exists;

%	\item[\upshape(ii)] every sequential weak cluster point of $\{x_n\}$ is in $C$.
%\end{enumerate}
%Then $\{x_n\}$ converges weakly to a point in $C$.
%\end{lemma}
\indent 
Following \cite{Led}, let $(\Omega, {\EuScript{F}},\mathsf{P})$ be a probability space. A $\mathcal H$-valued random variable is a measurable function $X:\Omega \to \mathcal H$, where $\mathcal H$ is endowed with the Borel $\sigma$-algebra. We denote by $\sigma(X)$ the $\sigma$-field generated by $X$. The expectation of a random variable $X$ is denoted by $\E[X]$. The conditional expectation of $X$ given a $\sigma$-field ${\EuScript{A}} \subset {\EuScript{F}}$ is denoted by $\mathsf{E}[X|{\EuScript{A}}]$. A $\mathcal H$-valued random process is a sequence $\{x_n\}$ of $\mathcal H$-valued random variables. The abbreviation a.s. stands for 'almost surely'.
\begin{lemma}{\upshape(\cite[Theorem 1]{robbins})} \label{lm1}
	Let $({\EuScript{F}}_n)_{n \in \mathbb N}$ be an increasing sequence of sub-$\sigma$-algebras of ${\EuScript{F}}$, let $(z_n)_{n\in\NN}, (\xi_n)_{n\in\NN}, (\zeta_n)_{n\in\NN}$ and $(t_n)_{n\in|NN}$ be $[0,+\infty[$-valued random sequences such that, for every $n \in \mathbb N$, $z_n,\xi_n,\zeta_n$ and $t_n$ are ${\EuScript{F}}_n$-measurable. Assume moreover that $\sum_{n \in \mathbb N}t_n<+\infty,\ \sum_{n \in \mathbb N}\zeta_n<+\infty \ a.s.$ and
	\begin{align*}
	(\forall n \in \mathbb N)\ \E[z_{n+1}|{\EuScript{F}}_n] \le (1+t_n)z_n+\zeta_n-\xi_n \ a.s.
	\end{align*}
	Then $z_n$ converges a.s. and $(\xi_n)$ is summable a.s.
\end{lemma}
\begin{corollary}\label{hq1}
	Let $({\EuScript{F}}_n)_{n \in \mathbb N}$ be an increasing sequence of sub-$\sigma$-algebras of ${\EuScript{F}}$, let $(x_n)_{n\in\NN}$ be $[0,+\infty[$-valued random sequences such that, for every $n \in \mathbb N$, $x_{n-1}$ is $\EuScript F_n$-measurable and \begin{align} \label{dks}
	\sum \limits_{n \in \NN} \E[x_n|\EuScript F_n]<+\infty \ \ a.s
	\end{align}
	Then $\sum \limits_{n \in \NN} x_n<+\infty$ a.s
\end{corollary}
\begin{proof}
Let us set $$(\forall n\in\NN)\; z_n=\sum \limits_{k=1}^{n-1} x_k.$$
Then, $z_n$ is $\EuScript F_n$ measurable. Moreover,
$$\E[z_{n+1}|\EuScript F_n]=z_n+\E[x_n|\EuScript F_n]$$
Hence, it follows from Lemma \ref{lm1} and \eqref{dks} that $(z_n)_{n \in \NN}$ converges a.s.
\end{proof}

\noindent The following lemma can be viewed as direct consequence of 
\cite[Proposition 2.3]{Combettes2015}.
\begin{lemma} \label{lm2}
	Let $C$ be a non-empty closed subset of $\mathcal{H}$ and let $(x_n)_{n\in\NN}$ be a $\mathcal{H}$-valued random process.
	Suppose that, for every $x\in C$, $(\|x_{n+1}-x\|)_{n\in\NN}$ converges a.s. 
Suppose that the set of weak sequentially cluster points of $(x_n)_{n\in\NN}$ is a subset of C a.s. Then $(x_n)_{n\in\NN}$ converges weakly a.s. to a $C$-valued random vector.
\end{lemma}

\section{Algorithm and convergences} \label{s:main}
We propose the following algorithm, for solving  \eqref{p:primal-dual},
 that requires  only the unbiased estimations of the monotone, $\mu-$Lipschitzian operators $B$.
\begin{algorithm} \label{maina}
Let $(\gamma_n)_{n\in\NN}$ be a sequence in $\left]0,\pinf\right[$. Let
$ x_0,x_{-1}$ be $\HH$-valued, squared integrable random variables. Iterates
\begin{equation}  \label{algo:main}
 (\forall n\in\NN)\quad
\begin{array}{l}
\left\lfloor
\begin{array}{l}
y_n = 2x_n-x_{n-1}\\
\text{Finding $r_n$}: \E[r_n|\EuScript{F}_n] = By_n \\
x_{n+1} = J_{\gamma_n A}(x_n-\gamma_n r_n),
\end{array}
\right.\\[2mm]
\end{array}\\
\end{equation}
where $\EuScript{F}_n =\sigma(x_0,x_1,\ldots, x_n)$.
\end{algorithm}
\begin{remark} Here are some remarks.
\begin{enumerate}
    \item  Algorithm \ref{maina} is an extension of the reflected forward backward splitting in \cite{sva2} which itself recovers the projected reflected gradient methods for monotone variational inequalities in \cite{Malitsky15} as a special case. Further connections to existing works in the deterministic setting can be found in \cite{sva2} as well as \cite{Malitsky15}.
    \item In the special case when $A$ is a normal cone operator, $A= N_X$ for some non-empty closed convex set, the iteration \eqref{algo:main} reduces to the one in \cite{Cui16}. However, as we will see in Remark \ref{re:2}, our convergences results are completely different from that of \cite{Cui16}. 
    \item The proposed algorithm shares the same structure as the stochastic forward-backward splitting in \cite{Combettes2015,Combettes2016,Jota2}. The main advantage of \eqref{algo:main} is the monotonicity and Lipschitzianity of $B$ which is much weaker than cocoercivity assumption in  \cite{Combettes2015,Combettes2016,Jota2}.
    \item Under the current conditions on $A$ and $B$, an alternative method "Between forward-backward and forward-reflected-backward" for solving Problem \ref{p:primal-dual} is presented in \cite[Section 6]{Malitsky18b} which remains require to evaluate the operator $B$ as well as its unbiased estimations.
\end{enumerate}
\end{remark}
We first prove some lemmas which will be used in the proof of Theorem \ref{tr1} and Theorem 
\ref{t:2}.

\begin{lemma}  \label{l:main}  Let $(\xx_n)_{n\in\NN}$ 
and $(\yy_n)_{n\in\NN}$ be generated by \eqref{algo:main}.
Suppose that $\AAA$ is $\phi_A$-uniformly monotone and $B$ is 
$\phi_B$-uniformly monotone. Let $\xx\in\zer(\AAA+\QQ)$ and set
\begin{equation}
(\forall n\in\NN)\; \epsilon_{n} =2\gamma_n \big( \phi_A (\|x_{n+1}-x\|)+ \phi_A (\|x_{n+1}-x_n\|)+ \phi_B (\|y_n-x\|)\big).
\end{equation}
The following holds.
\begin{align}\label{e:mainesti}
 &\|x_{n+1}-x\|^2+\epsilon_n+(3-\frac{\gamma_n}{\gamma_{n-1}}) \|x_{n+1}-x_n\|^2+\frac{\gamma_n}{\gamma_{n-1}}\|x_{n+1}-y_n\|^2+2\gamma_n \scal{r_n-Bx}{x_{n+1}-x_n}\notag \\
     & \le \|x_n-x\|^2+2\gamma_n \scal{r_{n-1}-Bx}{x_n-x_{n-1}}+2\gamma_n \scal{r_{n-1}-r_n}{x_{n+1}-y_n}+\frac{\gamma_n}{\gamma_{n-1}}\|x_n-y_n\|^2\notag \\
     &\ +2 \gamma_n \scal{r_n-By_n}{x-y_n}.
\end{align}
\end{lemma}

\begin{proof}
Let $n\in\NN$ and $\xx\in\zer(\AAA+\QQ)$.
Set  
\begin{equation}\label{e:pp}
    \pp_{n+1} = \frac{1}{\gamma_n}(\xx_n-\xx_{n+1})- \sss_n.
\end{equation}
Then, by the definition of the resolvent,
\begin{equation}
p_{n+1} \in A x_{n+1}.
\end{equation}
Since $\AAA$ is $\nu_A$-uniformly monotone and $- \QQ\xx\in \AAA\xx$, we obtain 
\begin{align}
 \scal{\frac{x_n-x_{n+1}}{\gamma_n}-r_n+Bx}{x_{n+1}-x}  \ge \nu_A \phi_A(\|x_{n+1}-x\|),
 \end{align}
 which is equivalent to 
 \begin{align}
 \scal{x_n-x_{n+1}}{x_{n+1}-x}-\gamma_n\nu_A \phi_A(\|x_{n+1}-x\|) \ge \gamma_n \scal{r_n-Bx}{x_{n+1}-x} \label{maine1}.
\end{align}
Let us estimate the right hand side of \eqref{maine1}. 
Using  $y_n=2x_n-x_{n-1}$, we have
\begin{align}
 \scal{r_n-Bx}{x_{n+1}-x}&=\scal{r_n-Bx}{x_{n+1}-y_n}+\scal{r_n-By_n}{y_n-x}+\scal{By_n-Bx}{y_n-x} \notag\\
 &=\scal{r_n-Bx}{x_{n+1}-x_n}-\scal{r_n-Bx}{x_n-x_{n-1}}+\scal{r_n-By_n}{y_n-x}\notag \\
 &\ \ +\scal{By_n-Bx}{y_n-x}\notag \\
 &=\scal{r_n-Bx}{x_{n+1}-x_n}-\scal{r_{n-1}-Bx}{x_n-x_{n-1}}+\scal{r_{n-1}-r_n}{x_n-x_{n-1}}\notag \\
 &\ \ +\scal{r_n-By_n}{y_n-x}+\scal{By_n-Bx}{y_n-x} \notag \\
 &=\scal{r_n-Bx}{x_{n+1}-x_n}-\scal{r_{n-1}-Bx}{x_n-x_{n-1}}+\scal{r_n-r_{n-1}}{x_{n+1}-y_n}\notag \\
 &\ \ +\scal{r_{n-1}-r_n}{x_{n+1}-x_n}+\scal{r_n-By_n}{y_n-x}+\scal{By_n-Bx}{y_n-x} \label{maine2}
 \end{align}
Using the  uniform monotonicity of $A$ again, it follows from \eqref{e:pp} that
 \begin{align}
   \scal{\frac{x_n-x_{n+1}}{\gamma_n}-r_n-\frac{x_{n-1}-x_n}{\gamma_{n-1}}+r_{n-1}}{x_{n+1}-x_n} \ge \nu_A \phi_A(\|x_{n+1}-x_n\|),
  \end{align}
which is equivalent to 
\begin{align}
 \scal{r_{n-1}-r_n}{x_{n+1}-x_n}& \ge \nu_A \phi_A(\|x_{n+1}-x_n\|)+\dfrac{\|x_{n+1}-x_n\|^2}{\gamma_n}\notag \\
   &\hspace{5cm}+\scal{\frac{x_n-y_n}{\gamma_{n-1}}}{x_{n+1}-x_n}. \label{maine3}
 \end{align}
 We have 
\begin{equation} \label{maine4}
\begin{cases}
2\scal{x_n-y_n}{\xx_{n+1}-\xx_n}  = \|\xx_{n+1}-\yy_{n}\|^2 -\|\xx_n-\yy_n\|^2 - \|\xx_{n+1}-\xx_n\|^2\\
2\scal{x_{n}-x_{n+1}}{x_{n+1}-x} = \|x_n-x\|^2 - \|x_n-x_{n+1}\|^2-\|x_{n+1}-x\|^2. 
\end{cases}
\end{equation}
Therefore, we derive
 from \eqref{maine1}, \eqref{maine2}, \eqref{maine3} and \eqref{maine4}, and the uniform monotonicity of $B$ that
 \begin{align}
  &\|x_n-x\|^2 - \|x_n-x_{n+1}\|^2-\|x_{n+1}-x\|^2-2\gamma_n\nu_A \phi_A(\|x_{n+1}-x\|) \notag \\
     &\ge 2\gamma_n \big(\scal{r_n-Bx}{x_{n+1}-x_n}-\scal{r_{n-1}-Bx}{x_n-x_{n-1}} \big)+2 \gamma_n \nu_B \phi_B(\|y_n-x\|)\notag \\
     &\ +2\gamma_n \scal{r_n-r_{n-1}}{x_{n+1}-y_n}+2\gamma_n \nu_A \phi_A(\|x_{n+1}-x_n\|)+2\|x_{n+1}-x_n\|^2\notag \\
     &\ +\frac{\gamma_n}{\gamma_{n-1}}\big(\|\xx_{n+1}-\yy_{n}\|^2 -\|\xx_n-\yy_n\|^2 - \|\xx_{n+1}-\xx_n\|^2\big)+2\gamma_n \scal{r_n-By_n}{y_n-x}.
 \end{align}
Hence, 
 \begin{align}
     &\|x_{n+1}-x\|^2+\epsilon_n+(3-\frac{\gamma_n}{\gamma_{n-1}}) \|x_{n+1}-x_n\|^2+\frac{\gamma_n}{\gamma_{n-1}}\|x_{n+1}-y_n\|^2+2\gamma_n \scal{r_n-Bx}{x_{n+1}-x_n}\notag \\
     & \le \|x_n-x\|^2+2\gamma_n \scal{r_{n-1}-Bx}{x_n-x_{n-1}}+2\gamma_n \scal{r_{n-1}-r_n}{x_{n+1}-y_n}+\frac{\gamma_n}{\gamma_{n-1}}\|x_n-y_n\|^2\notag \\
     &\ +2 \gamma_n \scal{r_n-By_n}{x-y_n},
 \end{align}
 which proves \eqref{e:mainesti}.
\end{proof}

\noindent
We also have the following lemma where \eqref{e:Qy} was used in \cite{Malitsky15} as well as 
in \cite{sva2}.
\begin{lemma}\label{l:Qes} For every $n\in\NN$, we have following estimations
\begin{align} 
2\scal{By_{n-1}-By_n}{x_{n+1}-y_{n}} \le  \mu(1+\sqrt{2}) \|\yy_n-\xx_{n}\|^2 + \mu \|\xx_n-\yy_{n-1}\|^2+ \mu\sqrt{2} \|\yy_n-\xx_{n+1}\|^2,\label{e:Qy}
%&\leq \frac{\mu(1+\sqrt{2})}{2}\|x_n-y_n\|^2 +\frac{\mu}{2} \|x_n-y_{n-1}\|^2 + \frac{\mu\sqrt{2}}{2} \|x_{n+1}-y_n\|^2.
\end{align}
and 
\begin{align} \label{Tn}
   \TT_n&=  \frac{1}{\gamma_n} \|\xx_{n}-\xx\|^2 +\mu \|\xx_n-\yy_{n-1}\|^2+(\frac{1}{\gamma_{n-1}}+ \mu(1+\sqrt{2}))\|\xx_{n}-\xx_{n-1}\|^2
+2 \alpha_{n-1} \notag\\
&\geq \frac{1}{2\gamma_n} \|\xx_{n}-\xx\|^2,
\end{align}
where $\alpha_n = \scal{By_n-Bx}{x_{n+1}-x_n}$.
\end{lemma}
\begin{proof} Let $n\in\NN$.  We have
\begin{align*}
2 \scal{\QQ \yy_{n-1} -\QQ \yy_n}{\xx_{n+1}-\yy_n} &\le 
2\|x_{n+1}-y_n\|\|By_{n-1}-By_n\|\\
&\le 2\mu\|x_{n+1}-y_n\|\|y_{n-1}-y_n\|\\
&\le \dfrac{\mu}{\sqrt 2}\|y_n-y_{n-1}\|^2+\mu \sqrt 2 \|x_{n+1}-y_n\|^2\\
&=\dfrac{\mu}{\sqrt 2}  \|y_n-x_n+x_n-y_{n-1}\|^2 +\mu\sqrt 2 \|x_{n+1}-y_n\|^2\\
&\le \dfrac{\mu}{\sqrt 2} \big((1+\frac{1}{\sqrt 2-1})\|y_n-x_n\|^2+(1+\sqrt 2-1)\|x_n-y_{n-1}\|^2 \big )\\
&\ \ \ \ \ +\mu\sqrt 2 \|x_{n+1}-y_n\|^2\\
&= \mu (1+\sqrt 2)\|x_n-y_n\|^2+\mu\|x_n-y_{n-1}\|^2+\mu\sqrt 2 \|x_{n+1}-y_n\|^2.
\end{align*}
Since
$\alpha_{n-1} =\scal{\QQ\yy_{n-1}-\QQ\xx}{\xx_{n}-\xx_{n-1}}$, we obtain
\begin{align}
    2 |\alpha_{n-1}| &\le 2 \mu \|y_{n-1}-x\|\|x_n-x_{n-1}\| \notag \\
   &\le 2 \mu (\|x_n-y_{n-1}\|+\|x_n-x\|)\|x_n-x_{n-1}\|\notag \\
   &\le \mu (\|x_n-y_{n-1}\|^2+\|x_n-x\|^2+2 \|x_n-x_{n-1}\|^2). \label{alpha}
\end{align}
Therefore, we derive from \eqref{alpha} and the definition of $\TT_n$ that
\begin{align*}
     \TT_n & \ge \frac{1}{2\gamma_n} \|\xx_{n}-\xx\|^2+(\frac{1}{2\gamma_n}-\mu)\|\xx_{n}-\xx\|^2 +(\frac{1}{\gamma_{n-1}}+ \mu(-1+\sqrt{2}))\|\xx_{n}-\xx_{n-1}\|^2 \\
     &\ge \frac{1}{2\gamma_n} \|\xx_{n}-\xx\|^2,
\end{align*}
which proves \eqref{Tn}.
\end{proof}

\begin{theorem} \label{tr1} 
The following hold.
\begin{enumerate}
\item \label{tr1i}
Let $(\gamma_n)_{n\in\NN}$ be a nondecreasing sequence in $\left]0, \frac{\sqrt 2-1}{\mu}\right[$, satisfies $$\tau= \underset{n \in \NN}{\text{inf}}(\dfrac{2}{\gamma_n}-\dfrac{1}{\gamma_{n-1}}-\mu (1+\sqrt 2))>0 $$
In the setting of Algorithm \ref{algo:main}, assume that the following condition are satisfied for $\EuScript{F}_n=\sigma((x_k)_{0 \le k \le n})$
\begin{equation}\label{e:sum}
 \sum_{n\in\mathbb{N}}\mathsf{E}[\|r_n-By_n\|^2|\EuScript{F}_n]<
 +\infty\ \  a.s
 \end{equation}
 Then  $(x_n)$ converges weakly to a random varibale $\overline{x}\colon\Omega\to\zer(A+B)$ a.s.
 \item\label{tr1ii} Suppose that $\dom(A)$ is bounded, $A$ or $B$ is uniformly monotone. Let $(\gamma_n)_{n\in\NN}$ be a monotone decreasing 
 sequence in $\left]0, \frac{\sqrt 2-1}{\mu}\right[$ such that 
 \begin{equation}\label{e:Stand}
     (\gamma_n)_{n\in\NN}\in\ell_{2}(\NN)\backslash\ell_1(\NN) 
     \quad \text{and}\quad 
     \sum_{n\in\NN}\gamma^{2}_n \mathsf{E}[\|r_n-By_n\|^2|\EuScript{F}_n] < \infty\; \text{a.s}.
 \end{equation}
 Then $(x_n)_{n\in\NN}$ converges strongly a unique solution $\overline{x}$. 
 \end{enumerate}
\end{theorem}

\begin{proof}
\ref{tr1i}:
 Let $n\in\NN$.
 Applying Lemma \ref{l:main} with $\nu_A=\nu_B =0$, we have,
\begin{align} \label{e:Stas1}
&\frac{1}{\gamma_{n}}\|\xx_{n+1}-\xx\|^2 
 +\frac{1}{\gamma_{n-1}}\|\xx_{n+1}-\yy_n\|^2
+\big( \frac{3}{\gamma_n}-\frac{1}{\gamma_{n-1}}\big) \|\xx_n-\xx_{n+1}\|^2
+2 \delta_n
\notag\\
&\leq \frac{1}{\gamma_n} \|\xx_{n}-\xx\|^2 +\frac{1}{\gamma_{n-1}}\|\xx_{n}-\yy_n\|^2
+ 2\delta_{n-1}
+2 \scal{r_{n-1}-r_n}{x_{n+1}-y_n}+2\beta_n
\end{align}
where $$\begin{cases}
    \delta_n=\scal{r_n-Bx}{x_{n+1}-x_n}\\
    \beta_n=\scal{r_n-By_n}{x-y_n}
\end{cases}$$

Let $\chi$ be in $\left]0,\frac{\tau}{2}\right[$, it follows from 
the Cauchy Schwarz's inequality and \eqref{e:Qy} that
\begin{align}\label{e:Sta1}
2 \scal{\rr_{n-1}- r_n}{\xx_{n+1}-\yy_n}&=2\scal{r_{n-1}-By_{n-1}+By_{n-1}-By_n+By_n-r_n}{x_{n+1}-y_n}\notag \\ 
&\le \frac{\|r_{n-1}-By_{n-1}\|^2}{\chi}+\chi \|x_{n+1}-y_n\|^2+ \frac{\|r_{n}-By_{n}\|^2}{\chi}+\chi \|x_{n+1}-y_n\|^2\notag \\
&+\mu(1+\sqrt{2}) \|\yy_n-\xx_{n}\|^2 + \mu \|\xx_n-\yy_{n-1}\|^2+ \mu\sqrt{2} \|\yy_n-\xx_{n+1}\|^2
\end{align}
Hence, we derive from \eqref{e:Stas1} and \eqref{e:Sta1} that 
\begin{align}
    &\frac{1}{\gamma_{n}}\|\xx_{n+1}-\xx\|^2 
 +(\frac{1}{\gamma_{n-1}}-\mu \sqrt 2-2\chi)\|\xx_{n+1}-\yy_n\|^2
+\big( \frac{3}{\gamma_n}-\frac{1}{\gamma_{n-1}} \big) \|\xx_n-\xx_{n+1}\|^2+2\delta_n
\notag\\
&\leq \frac{1}{\gamma_n} \|\xx_{n}-\xx\|^2 +(\frac{1}{\gamma_{n-1}}+\mu (1+\sqrt 2))\|\xx_{n}-\yy_n\|^2+\mu \|x_n-y_{n-1}\|^2
+2\delta_{n-1}+2\beta_n\notag \\
&\ \ +\frac{\|r_{n-1}-By_{n-1}\|^2}{\chi}+\frac{\|r_{n}-By_{n}\|^2}{\chi}.
\end{align}
In turn, using $\gamma_n \le \gamma_{n+1}$ and $x_n-y_n=x_{n-1}-x_n$
\begin{align}
  &  \frac{1}{\gamma_{n+1}}\|\xx_{n+1}-\xx\|^2 
 +\mu \|\xx_{n+1}-\yy_n\|^2
+\big( \frac{3}{\gamma_n}-\frac{1}{\gamma_{n-1}}\big) \|\xx_n-\xx_{n+1}\|^2
+2 \delta_n \notag \\
& \le \frac{1}{\gamma_{n}} \|\xx_{n}-\xx\|^2+\mu \|\xx_{n}-\yy_{n-1}\|^2+\big( \frac{3}{\gamma_{n-1}}-\frac{1}{\gamma_{n-2}}\big) \|\xx_n-\xx_{n-1}\|^2+2 \delta_{n-1} +2\beta_n \notag \\ &-(\frac{1}{\gamma_{n-1}}-\mu (1+\sqrt 2)-2\chi)\|\xx_{n+1}-\yy_{n}\|^2-\big( \frac{2}{\gamma_{n-1}}-\frac{1}{\gamma_{n-2}}-\mu (1+\sqrt 2) \big) \|x_n-x_{n-1}\|^2\notag\\
&\hspace{9.9cm}
+\frac{\|r_{n-1}-By_{n-1}\|^2+\|r_n-By_n\|^2}{\chi}. \label{e:unn}
  \end{align}
 Let us set
  \begin{equation}\label{e:ed0}
      \theta_n=\frac{1}{\gamma_{n}} \|\xx_{n}-\xx\|^2+\mu \|\xx_{n}-\yy_{n-1}\|^2+\big( \frac{3}{\gamma_{n-1}}-\frac{1}{\gamma_{n-2}} \big) \|\xx_n-\xx_{n-1}\|^2+2 \delta_{n-1}+\frac{\|r_{n-1}-By_{n-1}\|^2}{\chi}.
  \end{equation}
  We have 
  \begin{align}
      2|\delta_{n-1}|&=2|\scal{r_{n-1}-By_{n-1}}{x_n-x_{n-1}}+2\scal{By_{n-1}-Bx}{x_n-x_{n-1}}|\notag \\
      &\le \frac{\|r_{n-1}-By_{n-1}\|^2}{\chi}+\chi \|x_n-x_{n-1}\|^2+2\mu \|y_{n-1}-x\|\|x_n-x_{n-1}\|\label{ealpha} \\
      &\le \frac{\|r_{n-1}-By_{n-1}\|^2}{\chi}+\chi \|x_n-x_{n-1}\|^2+2 \mu \big( \|x_n-y_{n-1}\|+\|x_n-x\| \big)\|x_n-x_{n-1}\|\notag \\
      &\le \frac{\|r_{n-1}-By_{n-1}\|^2}{\chi}+\chi \|x_n-x_{n-1}\|^2+\mu \big( \|x_n-y_{n-1}\|^2+\|x_n-x\|^2+2\|x_n-x_{n-1}\|^2 \big ) \notag \\
      &\hspace{-1cm}\Rightarrow 
      \theta_n \ge (\frac{1}{\gamma_n}-\mu)\|x_n-x\|^2+\big( \frac{3}{\gamma_{n-1}}-\frac{1}{\gamma_{n-2}}-\chi -2\mu \big) \|x_n-x_{n-1}\|^2 \ge \mu \|x_n-x\|^2 \ge 0 \label{etheta}
      \end{align}
   Moreover, it follows from 
  \eqref{algo:main} that
    \begin{equation}
      \E[\beta_n|\EuScript{F}_n] = 0.
  \end{equation} 
  Therefore, by taking the conditional expectation both sides of \eqref{e:unn} 
  with respect to $\EuScript{F}_n$, we obtain
  \begin{align} 
    \E[\theta_{n+1}|\EuScript{F}_n] &\le \theta_n -(\frac{1}{\gamma_{n-1}}- \mu  (\sqrt 2+1)-2\chi)\E [\|\xx_{n+1}-\yy_{n}\|^2|\EuScript{F}_n]\notag\\
    &-\big( \frac{2}{\gamma_{n-1}}-\frac{1}{\gamma_{n-2}}-\mu (1+\sqrt 2)\big) \|\xx_n-x_{n-1}\|^2
    + 2\frac{\E[\|r_{n}-By_{n}\|^2|\EuScript F_n]}{\chi}. \label{esu}
\end{align}
 It follows from our conditions on step sizes $(\gamma_{n})_{n\in\NN}$ that
 \begin{equation}
     \dfrac{1}{\gamma_{n-1}}- \mu  (\sqrt 2+1)-2\chi>0 \; \text{and}\;
      \dfrac{2}{\gamma_{n-1}}-\dfrac{1}{\gamma_{n-2}}-\mu (1+\sqrt 2)-\chi>0,
 \end{equation}

Now, in view of Lemma \ref{lm1}, we get 
\begin{equation}\label{e:d1}
    \theta_n \to \bar \theta \ \ \text{and}\ \ \|x_n-x_{n-1}\| \to 0 \ \ a.s.
\end{equation}
 From \eqref{e:sum} and Corollary \ref{hq1}, we have 
 \begin{equation}
     \sum \limits_{n \in \NN} \|r_{n-1}-By_{n-1}\|^2<+\infty \ \Rightarrow \|r_{n-1}-By_{n-1}\| \to 0 \ \  a.s \label{er}
 \end{equation}
Since $(\theta_n)_{n\in\NN}$ converges, it is bounded and therefore, using \eqref{etheta}, it follows that $(\|x_n-x\|)_{n\in\NN}$ and $(x_n)_{n\in\NN}$ are bounded. Hence $(y_n)_{n\in\NN}$ is also bounded. In turn, from \eqref{ealpha}, we derive 
\begin{align}\delta_{n-1} \to 0 \ \ a.s  \label{e:d2}
\end{align}
Moreover,
\begin{equation}\label{e:d3}
    \|x_n-y_n\|=\|x_n-x_{n-1}\| \to 0,\; \text{and}\; 
    \|x_n-y_{n-1}\| \le \|x_n-x_{n-1}\|+\|x_{n-1}-y_{n-1}\| \to 0.
\end{equation}
Therefore, we derive from \eqref{e:ed0}, \eqref{e:d1}, \eqref{er}, \eqref{e:d2}, \eqref{e:d3} and Lemma \ref{lm1}
that 
\begin{equation}
\text{ $(\|x_n-x\|)_{n \in \NN}$ converges a.s.}
\end{equation}
%Using Lemma \ref{lm2}, there exists $\Omega_1 \in \EuScript F$ such that %$P(\Omega_1)=1$ and for every $\omega \in \Omega_1$ and every $x \in %\text{zer}(A+B)$, $(\|x_n(\omega)-x\|)_{n \in \NN}$ converges.
%Let $\omega \in \Omega_1$ and 
Let $x^*$ be a weak cluster point of $(x_n)_{n\in\NN}$. 
Then, there exists a subsequence $(x_{n_k})_{k \in \mathbb N} $ which converges weakly to $x^*$ a.s. By \eqref{e:d3}, $ y_{n_k} \rightharpoonup x^* $ a.s.  Let us next set 
\begin{equation}
z_n=(I+\gamma_n A)^{-1}(x_n-\gamma_n By_n).
\end{equation}
Then, since $J_{\gamma_n A}$ is nonexpansive, we have 
\begin{equation}\label{e:d5}
\|x_{n+1}-z_n\| \le \gamma_n \|By_n-r_n\| \to 0\; \text{a.s}.
\end{equation}
It follows from
$x_{n_k} \rightharpoonup x^*$ that $x_{n_k+1} \rightharpoonup  x^*$ and hence from \eqref{e:d5} that $ z_{n_k}(\omega) \rightharpoonup x^*(\omega)$.
%Let $(v,u)\in Gr(A+B)$, that is, $u-Bv\in Av$ and we have
Since 	$z_{n_k}=(I+\gamma_{n_k}A)^{-1}(x_{n_k}-\gamma_{n_k}B y_{n_k})$, we have
	\begin{equation}
	\frac{ x_{n_k}-z_{n_k}}{\gamma_{n_k}}- B y_{n_k} + B z_{n_k}\in (A+B)z_{n_k},
	\end{equation}
%	which implies that
%	\begin{equation*}
%	\dfrac{1}{\gamma_{n_k}}(x_{n_k}-z_{n_k}-\gamma_{n_k} By_{n_k})\in Az_{n_k}.
%	\end{equation*}
%	By the maximal monotonicity of $A$, we have
%	\begin{equation*}
%	\scal{ v-z_{n_k}}{ u-Bv-\dfrac{1}{\gamma_{n_k}}\left(x_{n_k}-z_{n_k}-\gamma_{n_k}By_{{n_k}}\right)}\geq 0.
%	\end{equation*}
%	Therefore, we obtain
%	\begin{align}\label{pt19}
%	\scal{ v-z^{n_k}}{u} \geq  &\scal{ v-z_{n_k}}{ Bv+\dfrac{1}{\gamma_{n_k}}\left(x_{n_k}-z_{n_k}-\gamma_{n_k} By_{{n_k}}\right)}\notag\\
%	&=\scal{v-z_{n_k}}{ Bv-By_{n_k}}+\scal{v-z_{n_k}}{ \dfrac{1}{\gamma_{n_k}}\left(x_{n_k}-z_{n_k}\right)}\notag\\
%	&=\scal{ v-z_{n_k}}{ Bv-Bz_{n_k}}+\scal{v-z_{n_k}}{ Bz_{n_k}-By_{n_k}}+\scal{v-z_{n_k}}{ \dfrac{1}{\gamma_{n_k}}\left(x_{n_k}-z_{n_k}\right)}\notag\\
%	&\geq \scal{ v-z_{n_k}}{ Bz_{n_k}-By_{n_k}}+\scal{v-z_{n_k}}{ \dfrac{1}{\gamma_{n_k}}\left(x_{n_k}-z_{n_k}\right)} \label{e:d6}
%	\end{align}
From \eqref{e:d3} and \eqref{e:d5}, we have
\begin{align}
    \lim_{k \to \infty} \|x_{n_k}-z_{n_k}\|=\lim_{k \to \infty} \|y_{n_k}-z_{n_k}\|=0
\end{align}
Since $B$ is $\mu$-Lipschitz and $(\gamma_n)_{n\in\NN}$ is bounded away from $0$, it follows that 
\begin{equation}
\frac{ x_{n_k}-z_{n_k}}{\gamma_{n_k}}- B y_{n_k} + B z_{n_k} \to 0 \quad{a.s.}
\end{equation}
%$(x_n)_{n \in \NN}$ is bounded, so $(z_n)_{n \in \NN}$ is too. Therefore, from \eqref{e:d6}, using $(\gamma_n)_{n \in \NN}$ is nondecreasing and $B$ is Lipschitz, we derive
%\begin{align}
    %\scal{v-x^*(\omega)}{u}=\lim_{k \to \infty} \scal{ v-z^{n_k}}{u} \ge 0
%\end{align}
Using  \cite[Corollary 25.5]{livre1}, the sum $A+B$ is  maximally monotone and hence, its graph is closed in $\HH^{weak}\times\HH^{strong}$ \cite[Proposition 20.38]{livre1}. 
Therefore, $0\in (A+B) x^*$ a.s., that is $x^*\in \text{zer}(A+B)$ a.s. By Lemma \ref{lm2}, the sequence $(x_n)_{n\in\NN}$ converges weakly to $\bar x \in \text{zer}(A+B)$ and the proof is complete

\noindent\ref{tr1ii}
It follows from Lemma \ref{l:main} and  \eqref{e:Qy} that
\begin{align} \label{e:tt1}
&\|\xx_{n+1}-\xx\|^2 
 +\frac{\gamma_n}{\gamma_{n-1}}\|\xx_{n+1}-\yy_n\|^2
+\big( 3-\frac{\gamma_n}{\gamma_{n-1}}\big) \|\xx_n-\xx_{n+1}\|^2
+2 \gamma_n\alpha_n \notag\\
&\ \ +2\gamma_n \scal{r_n-By_n}{x_{n+1}-x_n}+\epsilon_n
\notag\\
&\le \|x_n-x\|^2+2\gamma_n  \scal{r_{n-1}-By_{n-1}+By_{n-1}-By_n+By_n-r_n}{x_{n+1}-y_n}+\frac{\gamma_n}{\gamma_{n-1}}\|x_n-y_n\|^2\notag \\
&\ \ +2\gamma_n \alpha_{n-1}+2\gamma_n \scal{r_{n-1}-By_{n-1}}{x_n-x_{n-1}}+2\gamma_n \beta_n\notag \\
&\leq  \|\xx_{n}-\xx\|^2 
+\gamma_n \big( \mu(1+\sqrt{2}) \|\yy_n-\xx_{n}\|^2 + \mu \|\xx_n-\yy_{n-1}\|^2+ \mu\sqrt{2} \|\yy_n-\xx_{n+1}\|^2 \big)+\frac{\gamma_n}{\gamma_{n-1}}\|\xx_{n}-\yy_n\|^2 \notag \\
&\ +2 \gamma_n \alpha_{n-1}+2\gamma_n \scal{r_{n-1}-By_{n-1}}{x_n-x_{n-1}}+2\gamma_n \beta_n+2\gamma_n \scal{\sss_{n-1}- \BB\yy_{n-1}+By_n-r_n}{\xx_{n+1}-\yy_n},
\end{align}
For any $\eta\in\left]0,\frac{1-\gamma_0 \mu (1+\sqrt 2)}{10}\right[$, using 
the Cauchy Schwarz's inequality, we have 
\begin{equation} \label{bdt}
    \begin{cases} 2\gamma_n \scal{\sss_{n-1}- \BB\yy_{n-1}}{\xx_{n+1}-\yy_n} &\leq 
    \frac{\gamma_n^2}{\eta} \|r_{n-1}-By_{n-1}\|^2+\eta \|x_{n+1}-y_n\|^2\\
    2\gamma_n \scal{\sss_{n}- \BB\yy_{n}}{\xx_{n+1}-\yy_n} &\leq 
    \frac{\gamma_n^2}{\eta} \|r_{n}-By_{n}\|^2+\eta \|x_{n+1}-y_n\|^2\\
    2\gamma_n \scal{r_{n-1}-By_{n-1}}{x_n-x_{n-1}}& \le \frac{\gamma_n^2}{\eta} \|r_{n-1}-By_{n-1}\|^2+\eta \|x_{n}-x_{n-1}\|^2\\
    2\gamma_n\scal{r_{n}-By_{n}}{x_{n+1}-x_{n}}& \le \frac{\gamma_n^2}{\eta} \|r_{n}-By_{n}\|^2+\eta \|x_{n+1}-x_n\|^2
    \end{cases}
\end{equation}
We have \begin{equation} \label{itri}\|x_{n+1}-y_n\|^2 \le 2(\|x_{n+1}-x_n\|^2+\|x_n-y_n\|^2)
\end{equation}
Therefore, we derive from \eqref{e:tt1}, \eqref{bdt}, \eqref{itri}, the monotonic decreasing of
$(\gamma_n)_{n\in\NN}$ and $y_n-x_n=x_n-x_{n-1}$ that 
\begin{align}\label{e:tt2}
        &\|x_{n+1}-x\|^2+(\frac{\gamma_n}{\gamma_{n-1}}-\gamma_n \mu \sqrt 2)\|x_{n+1}-y_n\|^2+2 \|x_{n+1}-x_n\|^2+2\gamma_n \alpha_n+\epsilon_n \notag\\
    &\le \|x_n-x\|^2+\gamma_n \mu \|x_n-y_{n-1}\|^2+(1+\gamma_0 \mu (1+\sqrt 2)) \|x_n-x_{n-1}\|^2+2\gamma_{n-1}\alpha_{n-1}\notag \\
    &-2(\gamma_{n-1}-\gamma_n)\alpha_{n-1}+2\gamma_n \beta_n +2\frac{\gamma_{n-1}^2}{\eta} \|r_{n-1}-By_{n-1}\|^2+2\frac{\gamma_n^2}{\eta} \|r_{n}-By_{n}\|^2\notag \\
    &\ +5\eta \big( \|x_{n+1}-x_n\|^2+\|x_n-y_n\|^2 \big)
\end{align}
Since
$\dom (A)$ is bounded, 
there exists $M>0$ such that $(\forall n\in\NN)\;|\alpha_n| \le M$, and hence \eqref{e:tt2} 
implies that
\begin{align}
    &\|x_{n+1}-x\|^2+\gamma_n \mu \|x_{n+1}-y_n\|^2+(2-5\eta) \|x_{n+1}-x_n\|^2+2\gamma_n \alpha_n \notag \\
    & \le \|x_n-x\|^2+\gamma_{n-1} \mu \|x_n-y_{n-1}\|^2+(2-5\eta)\|x_n-x_{n-1}\|^2+2 \gamma_{n-1}\alpha_{n-1}\notag \\
    &-(\frac{\gamma_n}{\gamma_{n-1}}-\gamma_n \mu (\sqrt 2+1))\|x_{n+1}-y_n\|^2-(1-\gamma_0 \mu (1+\sqrt 2)-10\eta)\|x_n-x_{n-1}\|^2\notag \\
    &+2(\gamma_{n-1}-\gamma_n)M+2\frac{\gamma_{n-1}^2}{\eta} \|r_{n-1}-By_{n-1}\|^2+2\frac{\gamma_{n}^2}{\eta} \|r_{n}-By_{n}\|^2-\epsilon_n+2\gamma_n \beta_n.\label{e:ww3}
\end{align}
Let us set 
\begin{equation}
    p_n=\|x_n-x\|^2+\gamma_{n-1} \mu \|x_n-y_{n-1}\|^2+(2-5\eta)\|x_n-x_{n-1}\|^2+2 \gamma_{n-1}\alpha_{n-1}+2\frac{\gamma_{n-1}^2}{\eta} \|r_{n-1}-By_{n-1}\|^2.
\end{equation}
Then, by taking the conditional expectation with respect to $\EuScript{F}_n$ both 
sides of \eqref{e:ww3} and using $\E[r_n|\EuScript{F}_n] = By_n$, we get
\begin{align}
    \E [p_{n+1}&|\EuScript F_n] \le p_n-(\frac{\gamma_n}{\gamma_{n-1}}-\gamma_n \mu (\sqrt 2+1))\E[\|x_{n+1}-y_n\|^2|\EuScript F_n]-(1-\gamma_0 \mu (1+\sqrt 2)-10 \eta)\|x_n-x_{n-1}\|^2 \notag \\
    &+2(\gamma_{n-1}-\gamma_n)M+4\frac{\gamma_{n}^2}{\eta} \E[\|r_{n}-By_{n}\|^2|\EuScript F_n]-\E[\epsilon_n|\EuScript F_n]\label{e:ww5}
\end{align}
Note that, 
\begin{equation}
\begin{cases} \frac{\gamma_n}{\gamma_{n-1}}-\gamma_n \mu (\sqrt 2+1)>0,\\
1-\gamma_0 \mu (1+\sqrt 2)-10\eta>0,\\
\sum_{n \in \NN}(\gamma_{n-1}-\gamma_n)M =\gamma_0M
\end{cases}
\end{equation}
Similar to \eqref{Tn}, we have $p_n$ is a nonnegative sequence.
In turn,   Lemma \ref{lm1} and \eqref{e:ww5} give,
\begin{equation}\label{e:ww6}
 p_n \to \bar p, \ \ \|x_n-x_{n-1}\| \to 0\ \  \text{and}\ \  \sum \limits_{n \in \NN}\E[\epsilon_n|\EuScript F_n] <+\infty\ \  \ \ a.s.
\end{equation}
Using the same argument as the proof of \ref{tr1i},
\begin{equation}\label{e:ww8}
    \lim_{n\to\infty}\|x_n-x\|^2 = \bar p.\\
\end{equation}
Now, let us consider the case where
$A$ is $\phi_A-$uniformly monotone. We then derive from \eqref{e:ww6} that
\begin{equation}
    \sum \limits_{n \in \NN} \gamma_n \E[\phi_A(\|x_{n+1}-x\|)|\EuScript F_n]<+\infty,
\end{equation}
hence Corollary \ref{hq1} impies that
\begin{equation}
    \sum_{n \in \NN} \gamma_n \phi_A(\|x_{n+1}-x\|) <+\infty. \label{e:ww7}
\end{equation}
Since $  \sum_{n \in \NN} \gamma_n = \infty$, it follows from \eqref{e:ww7} that
$\underline{\lim} \phi_A(\|x_{n+1}-x\|)=0$.  Thus, there exists a subsequence $(k_n)_{n \in \NN}$ such that $\phi_A (\|x_{k_n}-x\|) \to 0$ and hence $\|x_{k_n}-x\| \to 0$. Therefore, by \eqref{e:ww8}, we obtain $x_n \to x$. We next consider that case when
 $B$ is $\phi_B-$uniformly monotone. 
 Since $y_n=2x_n-x_{n-1}$, by the triangle inequality,
 \begin{equation}
 \|x_n-x\|-\|x_n-x_{n-1}\| \le \|y_n-x\| \le \|x_n-x\|+\|x_{n-1}-x_n\|,
 \end{equation}
 and by \eqref{e:ww6}, we obtain $\lim \limits_{n \to \infty} \|y_n-x\|=\lim \limits_{n \to \infty} \|x_n-x\|$. Hence, by using the same argument as the case $A$ is uniformly monotone, we obtain $y_n \to x$ and hence $x_n \to x$. The proof of the theorem is complete.
\end{proof}

\begin{remark}
 For $0<\gamma<\dfrac{\sqrt 2-1}{\mu}$, $\dfrac{1}{2-\gamma \mu (1+\sqrt 2)}<c<1 $. Then for every $(\gamma_n)_{n \in \NN} \subset \left[c \gamma, \gamma\right]^\NN$, we have $\tau=\underset{n \in \NN}{\text{inf}}(\dfrac{2}{\gamma_n}-\dfrac{1}{\gamma_{n-1}}-\mu (1+\sqrt 2))>0$
\end{remark}

\begin{corollary}

%Set
%\begin{align}
 %   \mu=\text{max} \{\mu_0,\ldots,\mu_m\}+\sqrt{\sum \limits_{i=1}^m \|L_i\|^2}
%\end{align}
Let $\gamma \in \left]0,(\sqrt{2}-1)/\mu \right[$. 
Let
$ x_0,x_{-1}$ be $\HH$-valued, squared integrable random variables.
\begin{equation} 
 (\forall n\in\NN)\quad
\begin{array}{l}
\left\lfloor
\begin{array}{l}
y_n = 2x_n-x_{n-1}\\
\E[r_n|\EuScript{F}_n] = By_n \\
x_{n+1} = J_{\gamma A}(x_n-\gamma r_n).
\end{array}
\right.\\[2mm]
\end{array}\\
\end{equation}
Suppose that 
\begin{equation}
     \sum_{n\in\mathbb{N}}\mathsf{E}[\|r_n-By_n\|^2|\EuScript{F}_n]<+\infty\ \  a.s.\\
\end{equation}
Then  $(x_n)_{n\in\NN}$ converges weakly to a random variable $\overline{x}\colon\Omega\to\zer(A+B)$ a.s.
\end{corollary}

\begin{theorem}\label{t:2}
Suppose that $\AAA$ is $\nu$-strongly monotone. Define 
\begin{equation}\label{e:step}
  \forall n\in\NN)\quad  \gamma_{n} = \frac{1}{2\nu(n+1)}.
\end{equation}
Suppose that there exists a  constant $c$ such that 
\begin{equation}
(\forall n\in\NN)\; \mathsf{E}[\|\sss_n-\QQ\yy_n\|^2|\EuScript{F}_n]  \leq c.
\end{equation}
Then
\begin{equation}
(\forall n> n_0)\E\left[\|\xx_n-\overline{\xx}\|^2\right] = \mathcal{O}(\log(n+1)/(n+1)),
\end{equation}
where $n_0$ is the smallest integer such that $ n_0 > 4\nu^{-1}\mu(1+\sqrt{2})$.
\end{theorem}
\begin{proof} Let $n\in\NN$.
It follows from \eqref{e:step} that 
\begin{align}
    1+2\nu\gamma_n = \frac{2\nu(n+2)}{2\nu(n+1)}= \frac{\gamma_n}{\gamma_{n+1}}.
\end{align}
Set \begin{equation}\label{e:ro}
\begin{cases}
    \rho_{1,n}& =  \scal{\sss_n-\BB\yy_n}{\xx-\yy_n}+\scal{\sss_{n-1}- \BB\yy_{n-1}}{\xx_{n}-\xx_{n-1}}
-\scal{ \sss_n -\BB y_n}{\xx_{n+1}-\xx_{n}},\\
\rho_{2,n} &= \scal{r_{n-1}-By_{n-1}-r_n + By_n}{x_{n+1}-y_n},\\
\rho_{n}& = \rho_{1,n} +\rho_{2,n}.
\end{cases}
\end{equation}
Hence, by applying  Lemma \ref{l:main} with $\phi_B =0$ and $\phi_A = \nu |\cdot|^2$, we obtain 
\begin{align} 
 &(1+2\nu_A\gamma_n)\|x_{n+1}-x\|^2+(3+2\nu\gamma_n-\frac{\gamma_n}{\gamma_{n-1}}) \|x_{n+1}-x_n\|^2+\frac{\gamma_n}{\gamma_{n-1}}\|x_{n+1}-y_n\|^2+2\gamma_n\alpha_n \notag \\
     & \le \|x_n-x\|^2+2\gamma_n\alpha_{n-1} +2\gamma_n \scal{By_{n-1}- By_n}{x_{n+1}-y_n}+\frac{\gamma_n}{\gamma_{n-1}}\|x_n-y_n\|^2 +2\gamma_n\rho_{n}. 
\label{e:mainestim}
\end{align}
We derive from Lemma \ref{l:Qes} and \eqref{e:mainestim} that
\begin{align} 
&\frac{1}{\gamma_{n+1}}\|\xx_{n+1}-\xx\|^2 
 +(\frac{1}{\gamma_{n-1}}-\mu\sqrt{2})\|\xx_{n+1}-\yy_n\|^2
+\big( \frac{3}{\gamma_n}+2\nu-\frac{1}{\gamma_{n-1}}\big) \|\xx_n-\xx_{n+1}\|^2
+2 \alpha_n
\notag\\
&\leq \frac{1}{\gamma_n} \|\xx_{n}-\xx\|^2 +(\frac{1}{\gamma_{n-1}}+ \mu(1+\sqrt{2}))\|\xx_{n}-\xx_{n-1}\|^2+\mu \|\xx_n-\yy_{n-1}\|^2
+2 \alpha_{n-1}+2\rho_n.
\label{e:later1}
\end{align}
Now, using the definition of $\TT_n$, we can rewrite \eqref{e:later1} as
\begin{align}
     \TT_{n+1} &\leq \TT_n + 2\rho_n - (\frac{1}{\gamma_{n-1}}-\mu(\sqrt{2}+1))\|\xx_{n+1}-\yy_n\|^2\notag\\
    & -\big( \frac{2}{\gamma_n}+2\nu-\frac{1}{\gamma_{n-1}} -\mu(\sqrt{2}+1)\big) \|\xx_n-\xx_{n+1}\|^2
\end{align}
Let us rewrite $\rho_{2,n}$ as
\begin{align}
    \rho_{2,n} &= \scal{r_{n-1}-By_{n-1}}{x_{n+1}-x_n} 
   - \scal{r_{n-1}-By_{n-1}}{x_{n}-x_{n-1}}\notag\\
   &\quad -\scal{r_n - By_n}{x_{n+1}-x_n}+ \scal{r_n - By_n}{x_{n}-x_{n-1}},
\end{align}
which implies that
\begin{equation}\label{e:newrho}
    \rho_n = \scal{\sss_n-\BB\yy_n}{\xx-x_{n}} 
    + \scal{r_{n-1}-By_{n-1}}{x_{n+1}-x_{n}}- 2 
    \scal{r_n - By_n}{x_{n+1}-x_n}.
\end{equation}
Taking the conditional expectation with respect to $\EuScript{F}_n$, we obtain
    \begin{align}
     \E[\TT_{n+1}|\mathcal{F}_n] &\leq \TT_n + 2\E[\rho_n|\EuScript{F}_n] - (\frac{1}{\gamma_{n-1}}-\mu(\sqrt{2}+1))\E[\|\xx_{n+1}-\yy_n\|^2|\EuScript{F}_n]\notag\\
    & -\big( \frac{2}{\gamma_n}+2\nu-\frac{1}{\gamma_{n-1}} -\mu(\sqrt{2}+1)\big) \E[\|\xx_n-\xx_{n+1}\|^2|\EuScript{F}_n].\label{e:weekend}
\end{align}
By the definition of $\rho_n$ in \eqref{e:newrho}, we have 
\begin{align}
    2\E[\rho_n|\EuScript{F}_n] 
    &= 2\E[\scal{\sss_{n-1}- \BB\yy_{n-1}}{\xx_{n+1}-\xx_n}|\EuScript{F}_n]
    -4\E[\scal{\sss_{n}- \BB\yy_{n}}{\xx_{n+1}-\xx_n}|\EuScript{F}_n]
    \notag\\
    &\leq 2\gamma_{n-1}\E[\|\sss_{n-1}- \BB\yy_{n-1}\|^2|\EuScript{F}_n]
   + \frac{1}{2\gamma_{n-1}}\E[\| \xx_{n+1}-\xx_n\|^2|\EuScript{F}_n]\notag\\
   &\quad + 16\gamma_{n}\E[\|\sss_{n}- \BB\yy_{n}\|^2|\EuScript{F}_n]
    + \frac{1}{4\gamma_{n}}\E[\| \xx_{n+1}-\xx_n\|^2|\EuScript{F}_n]
\end{align}
In turn, it follows from \eqref{e:weekend} that 
    \begin{align}
     \E[\TT_{n+1}|\EuScript{F}_n] &\leq \TT_n  - (\frac{1}{\gamma_{n-1}}-\mu(\sqrt{2}+1))\E[\|\xx_{n+1}-\yy_n\|^2|\EuScript{F}_n]\notag\\
    & -\big( \frac{1}{4\gamma_n} -\mu(\sqrt{2}+1)\big) \E[\|\xx_n-\xx_{n+1}\|^2|\EuScript{F}_n] + 18\gamma_{n-1}c.
    \label{e:weekend1}
\end{align}
Note that for $n > n_0$, $\frac{1}{4\gamma_{n}}- \mu(1+\sqrt{2})\geq 0$, and hence taking 
expectation both the sides of \eqref{e:weekend1}, we obtain
\begin{equation}
    (\forall n > n_0)\; \E\left[\TT_{n+1}\right] \leq \E\left[\TT_{n_0}\right] 
    +c \sum_{k=n_0}^n \gamma_k, 
\end{equation}
which proves the desired result by invoking Lemma \ref{l:Qes}.
\end{proof}
\begin{remark}\label{re:2} We have some comparisons to existing work.
\begin{enumerate}
    \item Under the standard condition \eqref{e:Stand}, we obtain the strong almost sure convergence of the iterates, when $A$ or $B$ is uniformly monotone, as in the context of the stochastic forward-backward splitting \cite{Jota2}. In the general case, 
    to ensure the weak almost sure convergence, 
    we not only need the step-size bounded away from $0$ but also the summable condition in
    \eqref{e:sum}. These conditions were used in \cite{Combettes2015,Combettes2016,bang1,bang4}. 
    \item 
     In the case when $A$ is a normal cone in Euclidean spaces and the weak sharpness of $B$ is satisfied, as it was shown in \cite[Proposition1]{Cui16}, the strong 
    almost sure convergence of $(x_n)_{n\in\NN}$ is obtained under the condition \eqref{e:Stand}.
    %\begin{equation}\label{e:Stand}
     %  \E[\|\sss_{n}- \BB\yy_{n}\|^2|\mathcal{F}_n] \leq c^{2}_1\; \text{and}\;
      % (\gamma_n)_{n\in\NN} \in \ell^{2}(\NN)\backslash \ell^{1}(\NN).
    %\end{equation}
    Without imposing additional conditions on $B$ such as weak sharpness \cite{Cui16}, uniform monotonicity  \cite{Jota2}, the problem of proving the almost sure convergence 
    of the iterates under the condition \eqref{e:Stand} is still open.
    \item When $A$ is strongly monotone, we obtained the rate $\mathcal{O}(\log(n+1)/(n+1))$
    which is slower than the rate $\mathcal{O}(1/(n+1))$ of 
    the stochastic forward-backward splitting \cite{Jota2} and their extensions in \cite{bang3,bang8} . The main reason is the 
    monotonicity and Lipschitzianity of $B$ is weaker than the cocoercivity of $B$ as in 
    \cite{Jota2,bang1}.  
    \item  In the case when $A$ is a normal cone to a nonempty closed convex set $X$ in Euclidean spaces, the work in \cite{Cui16} obtained the rate $1/\sqrt{n}$ of the gap function defined by $X\ni x\mapsto \sup_{y\in X}\scal{By}{x-y}$. This rate of convergence was firstly established in \cite{Nem11} for solving variational inequalities with stochastic mirror-prox algorithm. Therefore, they differ from our results in the present paper.
\end{enumerate}
\end{remark}
We provide an generic special case which was widely studied in the stochastic optimization; see 
\cite{Atchade17,bang7,Duchi09,Bach14,Lan12, Kwok09} for instances.
\begin{corollary}\label{co:1}
Let $f\in\Gamma_0(\HH)$ and
let $h\colon\HH\to\RR$ be a convex differentiable function, with $\mu$-Lipschitz continuous gradient, given by an  expectation form $h(x)= \E_\xi [H(x,\xi)]$. In the expectation,  $\xi$ is a random vector whose probability distribution  is supported on a set  $\Omega_P \subset \RR^m$, and $H\colon \HH\times\Omega_p\to \RR$ is convex function with respect to the variable $x$. The problem is to
\begin{equation}\label{e:probapp}
\underset{  x\in \mathcal{H}}{\text{minimize}} \; f(x) +h(x),
\end{equation}
under the following assumptions: 
\begin{enumerate}
\item $\zer(\partial f +\nabla h)\not=\emp$.
\item It is possible to obtain independent and identically distributed (i.i.d.) samples  $(\xi_n)_{n\in\NN}$ of $\xi$.
\item { Given $(x,\xi)\in  \mathcal{H}\times\Omega_P$},
one can find a point $\nabla H(x,\xi)$ such that $ \E[\nabla H(x,\xi)] = \nabla h(x)$. 
\end{enumerate}
Let $(\gamma_n)_{n\in\NN}$ be a sequence in $\left]0,\pinf\right[$. Let
$ x_0,x_{-1}$ be in $\HH$.
\begin{equation}  \label{algo:app1}
 (\forall n\in\NN)\quad
\begin{array}{l}
\left\lfloor
\begin{array}{l}
y_n = 2x_n-x_{n-1}\\
x_{n+1} = \prox_{\gamma_n f}(x_n-\gamma_n \nabla H(y_n,\xi_n)).
\end{array}
\right.\\[2mm]
\end{array}\\
\end{equation}
Then, the following hold.
\begin{enumerate}
    \item If $f$ is $\nu$-strongly monotone, for some $\nu\in\left]0,+\infty\right[$,
and there exists a   constant $c$ such that 
\begin{equation}
 \E[\|\nabla H(y_n,\xi_n)-\nabla h(\yy_n)\|^2|\xi_0,\ldots, \xi_{n-1}] \leq c.
\end{equation}
Then,  for the learning rate $(\forall n\in\NN)\; \gamma_{n} = \frac{1}{2\nu(n+1)}$. We obtain
\begin{equation}
(\forall n> n_0)\E\left[\|\xx_n-\overline{\xx}\|^2\right] = \mathcal{O}(\log(n+1)/(n+1)),
\end{equation}
where $n_0$ is the smallest integer such that $ n_0 > 2\nu^{-1}\mu(1+\sqrt{2})$, and $\overline{x}$ is the unique solution to \eqref{e:probapp}.
\item If $f$ is not strongly monotone, let $(\gamma_n)_{n\in\NN}$ be a non-decreasing sequence in $\left]0, \frac{\sqrt 2-1}{\mu}\right[$, satisfies $\tau= \underset{n \in \NN}{\text{inf}}(\dfrac{2}{\gamma_n}-\dfrac{1}{\gamma_{n-1}}-\mu (1+\sqrt 2))>0 $ and 
\begin{equation}
 \sum_{n\in\mathbb{N}}\E[\nabla H(y_n,\xi_n)-\nabla h(\yy_n)\|^2|\xi_0,\ldots, \xi_{n-1}]<+\infty\ \  a.s
 \end{equation}
  Then  $(x_n)$ converges weakly to a random variable $\overline{x}\colon\Omega\to\zer(\partial f+\nabla h)$ a.s.
    \end{enumerate}
\end{corollary}
\begin{proof} The conclusions are followed from Theorem \ref{tr1} $\&$ \ref{t:2} where 
\begin{equation}
    A =\partial f, B =\nabla h,\; \text{and}\; (\forall n\in\NN)\; r_n = \nabla H(y_n,\xi_n).
\end{equation}
\end{proof}
\begin{remark} The algorithm \eqref{algo:app1} as well as the convergence results appear to be new. Algorithm \eqref{algo:app1} is different from 
the standard stochastic proximal gradient 
\cite{Atchade17,bang7,Duchi09,Bach14}  only  the evaluation of the stochastic gradients
at the reflections $(y_n)_{n\in\NN}$.
\end{remark}
\section{Ergodic convergences}
In this section, we focus on the class of primal-dual problem which was firstly investigated
in \cite{plc6}. This typical structured primal-dual framework covers a widely class of convex
optimization problems and it has found many applications to image processing, machine learning
\cite{plc6,plc7,Pesquet15,Cham16,Ouy13}.
We further exploit the duality nature of this framework 
to obtain a new stochastic primal-dual splitting method and focus on the ergodic convergence
of the primal-dual gap.
\begin{problem}\label{app2}
 Let $f\in\Gamma_0(\HH)$, $g\in\Gamma_0(\GG)$ and
let $h\colon\HH\to\RR$ be a convex differentiable function, with $\mu_h$-Lipschitz continuous gradient, given by an  expectation form $h(x)= \E_\xi [H(x,\xi)]$. In the expectation,  $\xi$ is a random vector whose probability distribution $P$ is supported on a set  $\Omega_p \subset \RR^m$, and $H\colon \HH\times\Omega\to \RR$ is convex function with respect to the variable $x$.
Let $\ell\in\Gamma_0(\GG)$ be a convex differentiable function with $\mu_\ell$-Lipschitz continuous gradient, and given by an  expectation form $\ell(v)= \E_\xi [L(v,\xi)]$. In the expectation,  $\zeta$ is a random vector whose probability distribution  is supported on a set $\Omega_D \subset \RR^d$, and $L\colon \GG\times\Omega_D\to \RR$ is convex function with respect to the variable $v$. Let $K\colon\HH\to\GG$ be a bounded linear operator. 
The primal problem is to 
\begin{align}
\underset{  x\in \mathcal{H}}{\text{minimize}} \; h(x)+(\ell^*\vuo g)(Kx)+f(x), \label{e:primalap}
\end{align}
and the dual problem is to 
\begin{align}
\underset{  v\in \mathcal{G}}{\text{minimize}} \; (h+f)^*(-K^*v)+g^*(v) +\ell(v), \label{e:dualap}
\end{align}
under the following assumptions: 
\begin{enumerate}
\item There exists a point $(x^\star,v^\star)\in\HH\times\GG$ such that the primal-dual gap function defined by
\begin{align}
G:&\mathcal{H}\times\mathcal{G} \to \mathbb R \cup \{-\infty,+\infty\} \notag\\
&(x,v) \mapsto h(x)+f(x)+\scal{Kx}{v} -g^*(v) -\ell(v) \label{s}
\end{align}
verifies the following condition:
\begin{align}\label{e:saddle}
\big(\forall x\in\HH\big)(\big(\forall v\in\GG\big)\;
G(x^\star,v) \leq G(x^\star,v^\star) \leq G(x,v^{\star}),
\end{align}
\item It is possible to obtain independent and identically distributed (i.i.d.) samples  $(\xi_n,\zeta_n)_{n\in\NN}$ of $(\xi,\zeta)$.
\item Given $(x,v,\xi,\zeta)\in \HH\times\GG\times\Omega_P\times\Omega_D$,
one can find a point $(\nabla H(x,\xi), \nabla L(v,\xi))$ such that 
\begin{equation}
\E_{(\xi,\zeta)}[(\nabla H(x,\xi),\nabla L(v,\zeta))] = (\nabla h(x),\nabla \ell(v)). 
\end{equation}
\end{enumerate}

\noindent Using the standard technique as in \cite{plc6}, 
we derive from \eqref{algo:app1} the following
stochastic primal-dual splitting method,
Algorithm \ref{al:spd}, for solving Problem \ref{app2}. The weak almost sure convergence and the convergence in expectation of the resulting algorithm can be derived easily from Corollary \ref{co:1} and hence we omit them here.

\begin{algorithm}\label{al:spd} 
Let $(x_0,x_{-1})\in\HH^2$ and  $(v_0,v_{-1})\in\GG^2$. Let 
$(\gamma_n)_{n\in\NN}$ be a non-negative sequence. Iterates
\begin{equation}\label{e:spd}
\begin{array}{l}
\operatorname{For}\; n=0,1\ldots,\\
\left\lfloor
\begin{array}{l}
	y_n=2x_n-x_{n-1}\\
	u_n = 2v_n-v_{n-1}\\
	x_{n+1}=\text{prox}_{\gamma_n f}(x_n-\gamma_n \nabla H(y_n,\xi_n)  -\gamma_n K^* u_n)\\
	v_{n+1}=\text{prox}_{\gamma_n g^*}(v_n-\gamma_n\nabla L(u_n,\zeta_n) + \gamma_n Ky_n)\\
	\end{array}
\right.\\[2mm]
\end{array}\
	\end{equation}
\end{algorithm}
\begin{theorem} Let $x_0=x_{-1}$, $v_0=v_{-1}$. 
Set $\mu = 2\max\{\mu_h,\ \mu_\ell\}+\|K\|$, let $(\gamma_n)_{n\in\NN}$ be a decreasing sequence in $\left]0, \frac{1}{2\mu}\right[$ such that
\begin{equation}\label{e:sumgam}
    \; e_0=\sum_{n\in\NN}\gamma_{n}^{2}\E\left[\|\nabla H(y_n,\xi_n)-\nabla h(y_n)\|^2+\|\nabla L(u_n,\zeta_n)-\nabla \ell(u_n)\|^2\right] < \infty.
\end{equation} 
For every $N\in\NN$, define 
\begin{equation}
	\hat{x}_N= \bigg(\sum_{n=0}^N\gamma_n x_{n+1}\bigg)/\bigg(\sum_{n=0}^N\gamma_n\bigg)\; \text{and}\;
	\hat {v}_N=\bigg(\sum_{n=0}^N \gamma_nv_{n+1}\bigg)/\bigg(\sum_{n=0}^N\gamma_n\bigg).
	\end{equation}
	Assume that $\dom f$ and $\dom g^*$ are bounded.
Then the following holds:
\begin{equation}
\E[G(\hat{x}_{N},v)-G(x,\hat{v}_{N})] \leq \bigg(\frac12\|(x_0,v_0)-(x,v)\|^2+ \gamma_0 c(x,v) +e_0\bigg)\bigg/\bigg(\sum_{k=0}^N\gamma_k\bigg)^{-1}.
\end{equation}
where 
\begin{equation}
c(x,v)= \|K\|\underset{n \in \NN}{\text{sup}} \{ 
\mathsf{E}\left[|\scal{ x_{n+1}-x}{v_{n+1}-v_n}|\right] +\mathsf{E}\left[|\scal{x_{n+1}-x_n}{v_{n+1}-v}|\right] \} < \infty. \label{e:conc}
\end{equation}
\end{theorem}

\begin{proof}
We first note that \eqref{e:conc} holds because of the boundedness of 
$\dom f$ and $\dom g^*$.
Since $\ell$ is a convex, differentiable function with $\mu_\ell$-Lipschitz continuous gradient, using the descent lemma,  
\begin{equation}
\ell(u) \leq \ell(q) + \scal{\nabla \ell(q)}{u-q} + \frac{\mu_\ell}{2} \|u-q\|^2. \label{e:descent1}
\end{equation}
Since $\ell$ is convex, $\ell(q) \leq \ell(w) + \scal{\nabla \ell(q)}{q-w}$. Adding this inequality to  \eqref{e:descent1}, we obtain 
\begin{equation}
\ell(u) \leq \ell(w) + \scal{\nabla \ell(q)}{u-w} + \frac{\mu_\ell}{2} \|u-q\|^2 \label{e:descent2}.
\end{equation}
In particular, applying \eqref{e:descent2} with $u=v_{n+1}$, $w=v$ and $q=u_n$, we get
\begin{equation}\label{e:rr1}
    \ell(v_{n+1})\leq \ell(v) +
    \scal{ \nabla \ell(u_n)}{v_{n+1}-v}+\frac{\mu_\ell}{2}\|v_{n+1}-u_n\|^2.
\end{equation}
Moreover, it follows from 
\eqref{e:spd} that 
\begin{equation}
-(v_{n+1}-v_n+\gamma_n \nabla L(u_n,\zeta_n)-\gamma_n K y_n) \in
\gamma_n\partial g^*(v_{n+1}),
\end{equation}
and hence, using the convexity of $g^*$, 
\begin{equation}
    g^*(v)-g^*(v_{n+1})\geq
    \frac{1}{\gamma_n}\scal{v_{n+1}-v}{v_{n+1}-v_n+\gamma_n \nabla L(u_n,\zeta_n)-\gamma_n K y_n}.\label{e:rr2}
\end{equation}
Therefore, we derive from \eqref{e:rr1}, \eqref{e:rr2} and \eqref{s} that 
\begin{align}
G(x_{n+1},v)&-G(x_{n+1},v_{n+1})=\scal{K x_{n+1}}{v-v_{n+1}}  -g^*(v)+g^*(v_{n+1})-\ell(v)+\ell(v_{n+1}) \notag \\
& \le \scal{ K x_{n+1}}{v-v_{n+1}}+\dfrac{1}{\gamma_n }\scal{v-v_{n+1}}{v_{n+1}-v_n+\gamma_n \nabla L(u_n,\zeta_n)-\gamma_n K y_n } \notag \\ 
&\hspace{6cm}+\scal{ \nabla \ell(u_n)}{v_{n+1}-v}+\frac{\mu_\ell}{2}\|v_{n+1}-u_n\|^2 \notag \\
&=\scal {K(x_{n+1}-y_n)}{v-v_{n+1}} +\dfrac{1}{\gamma_n} \scal{v-v_{n+1}}{v_{n+1}-v_n}+ \frac{\mu_\ell}{2}\|v_{n+1}-u_n\|^2 \notag \\
&\hspace{6.7cm} + \scal{ \nabla \ell(u_n)-\nabla L(u_n,\zeta_n)}{v_{n+1}-v}. \label{1}
\end{align}
By the same way, since $h$ is convex differentiable with 
$\mu_h$-Lipschitz gradient, we have 
\begin{equation}
   h(x_{n+1})-h(x)
   \leq \scal{ \nabla h(y_n)}{x_{n+1}-x} +\dfrac{\mu_h}{2}\|x_{n+1}-y_n\|^2.
\end{equation}
Moreover,  it follows from 
\eqref{e:spd} that
\begin{equation}
   -( x_{n+1}-x_n+\gamma_n \nabla H(y_n,\xi_n)+\gamma_n K^*(u_n))
    \in\gamma_n\partial f(x_{n+1}),
\end{equation}
and hence, by the convexity of $f$,
\begin{equation}
    f(x_{n+1})-f(x)\leq \frac{1}{\gamma_n} \scal{x-x_{n+1}}{x_{n+1}-x_n+\gamma_n \nabla H(y_n,\xi_n)+\gamma_n K^*(u_n)}.
\end{equation}
In turn, using the definition of $G$ as in \eqref{e:saddle}, we have
\begin{align}
G(x_{n+1},v_{n+1})-G(x,v_{n+1})&=h(x_{n+1})-h(x)+\scal{ K(x_{n+1}-x)}{v_{n+1}}+f(x_{n+1})-f(x) \notag \\
&\le \scal{ \nabla h(y_n)}{x_{n+1}-x} +\dfrac{\mu_h}{2}\|x_{n+1}-y_n\|^2+\scal{K(x_{n+1}-x)}{v_{n+1}} \notag\\
&\ \ +\frac{1}{\gamma_n} \scal{x-x_{n+1}}{x_{n+1}-x_n+\gamma_n \nabla H(y_n,\xi_n)+\gamma_n K^*(u_n)} \notag\\
&=\scal{ K(x_{n+1}-x)}{v_{n+1}-u_n}+\frac{1}{\gamma_n} \scal{ x-x_{n+1}}{x_{n+1}-x_n} \notag \\
&\ \ +\dfrac{\mu_h}{2}\|x_{n+1}-y_n\|^2 
 +\scal{ \nabla h(y_n)-\nabla H(y_n,\xi_n)}{x_{n+1}-x}. \label{2}
\end{align}
Let us set
\begin{equation}
    \begin{cases}
    \bar x_{n+1}&=\text{prox}_{\gamma_n f}(x_n-\gamma_n \nabla h(y_n)-\gamma_n K^*(u_n)),\\
    \bar v_{n+1}&=\text{prox}_{\gamma_n g^*}(v_n-\gamma_n \nabla \ell(u_n)+\gamma_n K(y_n)).
    \end{cases}
\end{equation}
Then, using the Cauchy Schwarz's inequality and the nonexpansiveness of $\prox_{\gamma_n f}$, we obtain
\begin{align}
  &\scal{\nabla h(y_n)-\nabla H(y_n,\xi_n)}{x_{n+1}-x} \notag\\
&\quad = \scal{\nabla h(y_n)-\nabla H(y_n,\xi_n)}{x_{n+1}-\bar x_{n+1}}+\scal{\nabla h(y_n)-\nabla H(y_n,\xi_n)}{\bar x_{n+1}-x} \notag \\
&\ \ \ \le \|\nabla H(y_n,\xi_n)-\nabla h(y_n)\|\|x_{n+1}-\bar x_{n+1}\|+\scal{\nabla h(y_n)-\nabla H(y_n,\xi_n)}{\bar x_{n+1}-x} \notag \\
&\ \ \ \le \gamma_n \|\nabla H(y_n,\xi_n)-\nabla h(y_n)\|^2+\scal{\nabla h(y_n)-\nabla H(y_n,\xi_n)}{\bar x_{n+1}-x}. \label{4}
\end{align}
By the same way,
\begin{align}
& \scal{\nabla \ell(u_n)-\nabla L(u_n,\zeta_n)}{v_{n+1}-v} \notag\\
 &\quad \le   \gamma_n \|\nabla L(u_n,\zeta_n)-\nabla \ell(u_n)\|^2+\scal{\nabla \ell(u_n)-\nabla L(u_n,\zeta_n)}{\bar v_{n+1}-v}\label{4'}.
\end{align}
It follows from
 \eqref{1}, \eqref{2} and \eqref{4}, \eqref{4'} that 
\begin{align}
&G(x_{n+1},v)-G(x,v_{n+1}) \notag \\
&\quad \le \dfrac{1}{\gamma_n} \big(\scal{v-v_{n+1}}{v_{n+1}-v_n}+\scal{x_n-x_{n+1}}{x_{n+1}-x} \big)+\dfrac{\mu_h}{2}\|x_{n+1}-y_n\|^2\notag \\
&\quad\quad +\scal{K(x_{n+1}-y_n)}{v-v_{n+1}}+\scal{K(x_{n+1}-x)}{v_{n+1}-u_n}
+\frac{\mu_\ell}{2}\|v_{n+1}-u_n\|^2\notag\\
&\quad\quad  +\scal{\nabla h(y_n)-\nabla H(y_n,\xi_n)}{\bar x_{n+1}-x }
+\scal{\nabla \ell(u_n)-\nabla L(u_n,\zeta_n)}{\bar v_{n+1}-v}\notag\\
&\quad\hspace{3cm} +\gamma_n \|\nabla H(y_n,\xi_n)-\nabla h(y_n)\|^2+\gamma_n \|\nabla L(u_n,\zeta_n)-\nabla \ell(u_n)\|^2,
\end{align}
which is equivalent to
\begin{align}
     \gamma_n& \big(G(x_{n+1},v)-G(x,v_{n+1})\big) \notag \\
&\le \big(\scal{v-v_{n+1}}{v_{n+1}-v_n}+\scal{x_n-x_{n+1}}{x_{n+1}-x} \big)+\dfrac{\mu_h \gamma_n}{2}\|x_{n+1}-y_n\|^2\notag \\
&\ \ +\gamma_n \big(\scal{K(x_{n+1}-y_n)}{v-v_{n+1}}+\scal{K(x_{n+1}-x)}{v_{n+1}-u_n} \big)+\frac{\mu_\ell \gamma_n}{2}\|v_{n+1}-u_n\|^2\notag\\
&\ \ +\gamma_n^2 \big( \|\nabla H(y_n,\xi_n)-\nabla h(y_n)\|^2+ \|\nabla L(u_n,\zeta_n)-\nabla \ell(u_n)\|^2 \big) \notag \\
&\ \ +\gamma_n \big(\scal{\nabla h(y_n)-\nabla H(y_n,\xi_n)}{\bar x_{n+1}-x}+\scal{\nabla \ell(u_n)-\nabla L(u_n,\zeta_n)}{\bar v_{n+1}-v}\big). \label{g}
\end{align}
For simple, set $\mu_0 = \max\{\mu_h,\mu_\ell\}$ and let us define some notations in the space $\HH\times\GG$ where the scalar product and the associated norm are defined in the normal manner,
\begin{equation}
    \begin{cases}
            \mathsf{x} &= (x,v),\quad\;
        \mathsf{x}_n = (x_n,v_n),\quad\;
                \mathsf{y}_n = (y_n,u_n),\; \overline{\mathbf{x}}_n= (\overline{x}_n,\overline{v}_n),\\
          \mathsf{r}_n &= (\nabla H(y_n,\xi_n),\nabla L(u_n,\zeta_n) ),\\
          \mathsf{R}_n &= (\nabla h(y_n), \nabla \ell(u_n)),\\
    \end{cases}
\end{equation}
and 
\begin{equation}
    S\colon\HH\times\GG\to\HH\times\GG\colon (x,v)\mapsto (K^*v,-Kx).
\end{equation}
Then, one has $\|S\| = \|K\|$ and 
\begin{align}
  &\scal{K(x_{n+1}-y_n)}{v-v_{n+1}}+\scal{K(x_{n+1}-x)}{v_{n+1}-u_n}  \notag\\
   &\quad= \scal{S(\mathsf{x}_{n+1}-\mathsf{x}_n)}{\mathsf{x}_{n+1}-\mathsf{x}}
  - \scal{S(\mathsf{x}_{n}-\mathsf{x}_{n-1})}{\mathsf{x}_{n}-\mathsf{x}}
  - \scal{S(\mathsf{x}_n- \mathsf{x}_{n-1})}{\mathsf{x}_{n+1}-\mathsf{x}_n}\notag\\
  &\quad \leq d_{n+1}- d_n + \frac{\|K\|}{2} \big(\|\mathsf{x}_n- \mathsf{x}_{n-1})\|^2
  + \|\mathsf{x}_{n+1}-\mathsf{x}_n\|^2\big),\label{e:vd0}
\end{align}
where we set $d_n =\scal{S(\mathsf{x}_{n}-\mathsf{x}_{n-1})}{\mathsf{x}_{n}-\mathsf{x}}$.
Moreover, we also have
\begin{align}
   \scal{v-v_{n+1}}{v_{n+1}-v_n}&+\scal{x_n-x_{n+1}}{x_{n+1}-x}\notag\\
   & = \scal{\mathsf{x}_{n+1}-\mathsf{x}_n}{\mathsf{x}-\mathsf{x}_{n+1}}\notag\\
   &= \frac12 \|\mathsf{x}_n-\mathsf{x}\|^2-\frac12 \|\mathsf{x}_n-\mathsf{x}_{n+1}\|^2
   -\frac12 \|\mathsf{x}_{n+1}-\mathsf{x}\|^2.\label{e:vd1}
\end{align}
Furthermore, using the triangle inequality, we obtain
\begin{equation}
    \dfrac{\mu_h \gamma_n}{2}\|x_{n+1}-y_n\|^2 + \frac{\mu_\ell \gamma_n}{2}\|v_{n+1}-u_n\|^2 \leq \gamma_n\mu_0 \big(\|\mathsf{x}_n-\mathsf{x}_{n+1}\|^2 +\|\mathsf{x}_n-\mathsf{x}_{n-1}\|^2 \big).
    \label{e:vd2}
\end{equation}
Finally, we can rewrite the two last terms in \eqref{g} as 
\begin{align}
   & \gamma_n^2 \big( \|\nabla H(y_n,\xi_n)-\nabla h(y_n)\|^2+ \|\nabla L(u_n,\zeta_n)-\nabla \ell(u_n)\|^2 \big) \notag \\
&\ \ +\gamma_n \big(\scal{\nabla h(y_n)-\nabla H(y_n,\xi_n)}{\bar x_{n+1}-x}+\scal{\nabla \ell(u_n)-\nabla L(u_n,\zeta_n)}{\bar v_{n+1}-v}\big)\notag\\
&= \gamma_{n}^2 \|\mathsf{r}_n-\mathsf{R}_n\|^2 
+ \gamma_n\scal{\mathsf{r}_n-\mathsf{R}_n}{\mathsf{x}-\bar{\mathsf{x}}_{n+1}}\label{e:vd3}
\end{align}
Therefore, inserting \eqref{e:vd0},  \eqref{e:vd1}, \eqref{e:vd2} and \eqref{e:vd3} into 
\eqref{g} and rearranging, we get
\begin{align}
       \gamma_n \big(G(x_{n+1},v)-&G(x,v_{n+1})\big)\leq 
       \frac12 \|\mathsf{x}_n-\mathsf{x}\|^2- \frac12 \|\mathsf{x}_{n+1}-\mathsf{x}\|^2 + \gamma_nd_{n+1}-\gamma_n d_n\notag\\
     &  -(\frac12 -\gamma_n\mu_0 -\frac{\gamma_n\|K\|}{2})\|\mathsf{x}_n-\mathsf{x}_{n+1}\|^2 
       +(\gamma_n\mu_0 +\frac{\gamma_n\|K\|}{2})\|\mathsf{x}_n-\mathsf{x}_{n-1}\|^2\notag\\
       &\hspace{5cm}+\gamma_{n}^2 \|\mathsf{r}_n-\mathsf{R}_n\|^2 
+ \gamma_n\scal{\mathsf{r}_n-\mathsf{R}_n}{\mathsf{x}-\bar{\mathsf{x}}_{n+1}}.\label{e:vd6}
\end{align}
Let us set 
\begin{equation}
    b_{n} = \frac12 \|\mathsf{x}_n-\mathsf{x}\|^2+ (\gamma_n\mu_0 +\frac{\gamma_n\|K\|}{2})\|\mathsf{x}_n-\mathsf{x}_{n-1}\|^2 -\gamma_n d_n.
\end{equation}
We have \begin{align*}
   & |\gamma_n d_n| \le \gamma_n \|K\| \|\mathsf{x}_n-\mathsf{x}\|\|\mathsf{x_n}-\mathsf{x}_{n-1}\| \le \frac{\gamma_n \|K\|}{2}\big(\|\mathsf{x}_n-\mathsf{x}\|^2+\|\mathsf{x_n}-\mathsf{x}_{n-1}\|^2 \big)\\
    &\Rightarrow b_n \ge 0 \ \forall n \in \NN
    \end{align*}
Then, we can rewrite \eqref{e:vd6} as
\begin{align}
       \gamma_n \big(G(x_{n+1},v)-&G(x,v_{n+1})\big)\leq b_{n}-b_{n+1}
       -(\frac12 -2\gamma_n\mu_0 -\frac{2\gamma_n\|K\|}{2})\|\mathsf{x}_n-\mathsf{x}_{n+1}\|^2  \notag\\
       &\hspace{1cm}+(\gamma_n-\gamma_{n+1})d_{n+1}+\gamma_{n}^2 \|\mathsf{r}_n-\mathsf{R}_n\|^2 
+ \gamma_n\scal{\mathsf{r}_n-\mathsf{R}_n}{\mathsf{x}-\bar{\mathsf{x}}_{n+1}}.\label{e:vd7}
\end{align}
Now, using our assumption, since $\overline{\mathsf{x}}_{n+1}$ is independent of $(\xi_n,\zeta_n)$,  we have 
\begin{equation}
    \mathsf{E}\left[\scal{\mathsf{r}_n-\mathsf{R}_n}{\mathsf{x}-\bar{\mathsf{x}}_{n+1}}|(\xi_0,\zeta_0),\ldots (\xi_{n-1},\zeta_{n-1}) \right] =0.
\end{equation}
Moreover, the condition on the learning rate gives 
 \begin{equation}
 \frac12 -2\gamma_n\mu_0 -\gamma_n\|K\| \geq 0\  \text{and}\; \gamma_n -\gamma_{n+1}\geq 0.
 \end{equation}
 Therefore, taking expectation both sides of \eqref{e:vd7}, we obtain
 \begin{align}
    \mathsf{E}[\gamma_n\big(G(x_{n+1},v)-&G(x,v_{n+1})\big)]
    \leq \mathsf{E}\left[b_{n}\right]-\mathsf{E}\left[b_{n+1}\right]
      + (\gamma_n-\gamma_{n+1})c(x,v) +\gamma_{n}^2   \mathsf{E}\left[\|\mathsf{r}_n-\mathsf{R}_n\|^2\right]. \label{e:vd8}
\end{align}
Now, for any $N\in\NN$,
summing \eqref{e:vd7} from $n=0$ to $n=N$ and invoking the convexity-concavity of $G$, we arrive at the desired result.
\end{proof}

\begin{remark} Here are some remarks.
\begin{enumerate}
\item To the best of our knowledge, this is first work establishing the rate convergence of the primal-dual gap for structure convex optimization involving infimal convolutions.
    \item The results presented in this Section are new even in the deterministic setting. In this case, by setting $\gamma_n \equiv\gamma$, our results share the same rate convergence $\mathcal{O}(1/ N)$ of the primal-dual gap as in \cite{te15}. While in the stochastic setting, our  results share the same rate convergence of the primal-dual gap as in \cite{bang4} under the same conditions on $(\gamma_n)_{n\in\NN}$ and variances as  in \eqref{e:sumgam}. However, the work in \cite{bang4} are limited to the case $\ell$ is a constant function.
\end{enumerate}

\end{remark}

\end{problem}

\end{document}